\documentclass{article}
\usepackage{fullpage}
\usepackage{amssymb, amsmath, amsthm, enumerate, math tools, tikz-cd}

\theoremstyle{plain}
\newtheorem{thm}{Theorem}
\theoremstyle{definition}
\newtheorem*{defn}{Definition}
\newtheorem{prop}[thm]{Proposition}
\newtheorem{cor}[thm]{Corollary}

\newtheorem{conj}[thm]{Conjecture}
\begin{document}
\title{On a theorem of Erd{\H o}s and Loxton}
\author{Noah Lebowitz-Lockard}
\maketitle

\begin{abstract} Let $a(n)$ be the number of partitions of $n$ of the form $a_1 + a_2 + \cdots + a_k$ where $a_{i + 1}$ is a proper divisor of $a_i$ for all $i < k$. Erd{\H o}s and Loxton showed that the sum of $a(n)$ over all $n \leq x$ is asymptotic to a constant multiple of $x^\rho$ where $s = \rho \approx 1.73$ is the unique solution to the equation $\zeta(s) = 2$ satisfying $s > 1$. In this note, we provide tight bounds on the value of this constant, though we do not find an exact formula for it. In addition, we write an explicit upper bound for $a(n)$.
\end{abstract}

\section{Introduction}

In this paper we consider partitions with a specific factorization-related property. We let $a(n)$ be the number of partitions of $n$ into distinct parts in which every part is a multiple of the next part. For example, $a(10) = 4$ because we have
\[10 = 9 + 1 = 8 + 2 = 6 + 3 + 1.\]
We also define $A(x)$ as the sum
\[\sum_{n \leq x} a(n).\]

It is straightforward to show that these factorizations are closely related to the number of ways one can express a number as a product of numbers greater than $1$. Though we discuss this correspondence in more detail later on, we write a few results about these products here.

\begin{defn} Let $g(n)$ be the number of ways to express the number $n$ as an ordered product of integers greater than $1$ (with $g(1) = 1$ for notational convenience). From here on, we refer to these products as \emph{factorizations} of $n$. (They are also called ``multiplicative partitions" in the literature.) We also let $G(x)$ be the sum of $g(n)$ over all $n \leq x$.
\end{defn}

In 1931, Kalm{\' a}r \cite{Kal} found an asymptotic formula for $G(x)$.

\begin{thm} As $x \to \infty$, we have
\[G(x) \sim -\frac{1}{\rho \zeta'(\rho)} x^\rho,\]
where $s = \rho \approx 1.73$ is the unique real solution to the equation $\zeta(s) = 2$ with $s > 1$.
\end{thm}

(Here, $\zeta$ is the Riemann zeta function. Note that $\zeta'(\rho)$ is negative, cancelling out the negative sign in front.) One may also prove this result by applying the Wiener-Ikehara Theorem \cite[Thm. 7.1]{P} to the Dirichlet series of $g(n)$. Ikehara \cite{I} refined the error term and Hwang \cite{Hw} later proved that
\[G(x) = -\frac{1}{\rho \zeta'(\rho)} x^\rho + O(x^\rho \exp(-c (\log \log x)^{(3/2) - \epsilon})),\]
where $\epsilon$ can be any positive number and $c := c(\epsilon)$ is a positive constant. (For the corresponding sum for \emph{unordered} factorizations, see \cite{O}.)

Erd{\H o}s and Loxton \cite[Thm. $2$]{ELox} proved the following result about $A(x)$.

\begin{thm} Define $\rho$ as we did in the asymptotic formula for $G(x)$. Then there exists a positive constant $c$ such that
\[A(x) \sim cx^\rho.\]
In addition, there does not exist a positive number $\epsilon$ for which $A(x) = cx^\rho + O(x^{\rho - \epsilon})$.
\end{thm}

Though Erd{\H o}s and Loxton proved this result, they do not provide any bounds on $c$. In this paper we prove the following inequality on $c$.

\begin{thm}\label{main} For every positive integer $k$, the value of $c$ lies between
\[-\frac{2}{\rho \zeta'(\rho)} \sum_{d_1, d_2, \ldots, d_k > 1} \frac{1}{(2 + d_1 + d_1 d_2 + \cdots + (d_1 d_2 \cdots d_k))^\rho}\]
and
\[-\frac{2}{\rho \zeta'(\rho)} \sum_{d_1, d_2, \ldots, d_k > 1} \frac{1}{(1 + d_1 + d_1 d_2 + \cdots + (d_1 d_2 \cdots d_k))^\rho}.\]
\end{thm}

In order to show that these sums converge, we note that
\begin{eqnarray*}
\sum_{d_1, d_2, \ldots, d_k > 1} \frac{1}{(1 + d_1 + d_1 d_2 + \cdots + (d_1 d_2 \cdots d_k))^\rho} & < & \sum_{d_1, d_2, \ldots, d_k > 1} \frac{1}{(d_1 d_2 \cdots d_k)^\rho} \\
& = & \left(\sum_{d = 2}^\infty \frac{1}{d^\rho}\right)^k \\
& = & (\zeta(\rho) - 1)^k \\
& = & 1^k \\
& = & 1.
\end{eqnarray*}
We may also observe that as $k$ increases the upper and lower bounds on $c$ get closer together. Because $d_i > 1$ for all $i$, we have $d_1 d_2 \cdots d_k \geq 2^k$ for all tuples $(d_1, d_2, \ldots, d_k)$. Therefore,
\[\frac{(2 + d_1 + d_1 d_2 + \cdots + (d_1 d_2 \cdots d_k))^\rho}{(1 + d_1 + d_1 d_2 + \cdots + (d_1 d_2 \cdots d_k))^\rho} < \left(1 + \frac{1}{2^k}\right)^\rho\]
for any such tuple. The ratio between the upper and lower sums is also at most $(1 + 2^{-k})^\rho$. If we take the limit as $k \to \infty$, then we obtain an expression for $c$.

\begin{cor} We have
\[c = -\frac{2}{\rho \zeta'(\rho)} \lim_{k \to \infty} \sum_{d_1, d_2, \ldots, d_k > 1} \frac{1}{(d_1 + d_1 d_2 + \cdots + (d_1 d_2 \cdots d_k))^\rho}.\]
\end{cor}

The sums in Theorem \ref{main} have closed forms in the $k = 1$ case, namely
\[-\frac{2(1 - 2^{-\rho} - 3^{-\rho})}{\rho \zeta'(\rho)} \leq c \leq -\frac{2(1 - 2^{-\rho})}{\rho \zeta'(\rho)}.\]
Unfortunately, I have been unable to obtain good numerical bounds on $c$. The previous inequality only gives us $c \in [0.349, 0.444]$.

If we remove the restriction that the parts have to be distinct, we obtain a completely different result.

\begin{thm}[{\cite[Thm. $3$]{ELox}}] Let $f(n)$ be the number of partitions of $n$ of the form $a_1 + a_2 + \cdots + a_k$ with $a_{i + 1} | a_i$ for all $i < k$, where the $a_i$'s are not necessarily distinct. Then,
\begin{align*}
\log f(n) = & \frac{1}{2 \log 2} (\log n - \log \log n)^2 + \left(\frac{1}{2} + \frac{1 + \log \log 2}{\log 2}\right) \log n - \left(1 + \frac{\log \log 2}{\log 2}\right) \log \log n \\
& + V\left(\frac{1}{\log 2} (\log n - \log \log n)\right) + o(1),
\end{align*}
where $V(t)$ is a periodic function with period $1$.
\end{thm}

de Bruijn \cite{dB} had previously proved this result holds for \emph{binary} partitions, i.e., partitions into powers of $2$, though he did not have the same periodic function $V(x)$. (Mahler \cite{Mah} had previously obtained a weaker asymptotic.)

In the next section, we prove Theorem \ref{main} and in Section $3$ we discuss some possibilities for future research.

\section{The proof}

Before proving our main result, we write a short argument relating $A(x)$ to $G(x)$. From here on, we let $b(n)$ be the number of partitions of $n$ counted by $a(n)$ with the additional restriction that the last part is not $1$. Note that the number of partitions in which the smallest part is $1$ is simply $b(n - 1)$ as we can simply remove the $1$ from our partition to obtain a partition of $n - 1$ into distinct parts in which every part is a multiple of the next part and the smallest part is not $1$. Therefore, $a(n) = b(n) + b(n - 1)$. If we let $B(x)$ be the sum of $b(n)$ over all $n \leq x$, then we have $A(x) = B(x) + B(x - 1)$, which is asymptotic to $2B(x)$. From here on, we consider $B(x)$ for notational convenience.

Consider the following map between partitions and factorizations. Let $a_1 + a_2 + \cdots + a_k$ be factorization of $n$ with $a_k > 1$. As usual, we suppose that $a_i$ is a multiple of $a_{i + 1}$ for all $i < k$. Then $(a_1/a_2)(a_2/a_3) \cdots (a_{k - 1}/a_k) (a_k)$ is a factorization of $a_1$. Note that this map is bijective. The inverse maps the factorization $d_1 d_2 \cdots d_\ell$ to the partition $(d_1 d_2 \cdots d_\ell) + (d_2 d_3 \cdots d_\ell) + \cdots + d_\ell$. Using this map, we can bound $B(n)$.

\begin{prop} For all $x$, we have $G(x/2) \leq B(x) \leq G(x)$.
\end{prop}

\begin{proof} Let $a_1 + a_2 + \cdots + a_k$ be a partition of some $n \leq x$. Then, $a_1 \leq n \leq x$. Applying our map turns this partition into a factorization of $a_1 \leq n \leq x$. Therefore, $B(x) \leq G(x)$.

Likewise, suppose that $a_1 + a_2 + \cdots + a_k$ is a partition with $a_1 \leq x/2$. Because $a_{i + 1}$ is a proper divisor of $a_i$, we have $a_{i + 1} \leq a_i/2$, giving us
\[a_1 + a_2 + \cdots + a_k \leq a_1 + (a_1/2) + \cdots + (a_1/2^{k - 1}) < 2a_1 \leq x.\]
So, every partition in which the first number is $\leq x/2$ corresponds to a factorization of some number $< x$. Hence, $B(x) \geq G(x/2)$.
\end{proof}

Using the asymptotic formula for $G(x)$, we can bound $c$. We have
\[-\frac{2}{2^\rho \rho \zeta'(\rho)} \leq c \leq -\frac{2}{\rho \zeta'(\rho)}.\]
Applying a more sophisticated version of the argument in our previous proof gives us our main result.

\begin{proof}[Proof of Theorem \ref{main}] Once again, we consider $B(x)$. Fix a positive integer $k$. The number of partitions of $n \leq x$ with our desired properties and at most $k$ parts is $n^{1 + o_k (1)}$. Consider one such partition $a_1 + a_2 + \cdots + a_\ell$ with $\ell \leq k$. There are $\lfloor x \rfloor$ possible values of $a_1$. However, if $i > 1$, then there are at most $d(a_1)$ possible values of $a_i$ because $a_i$ must divide $a_1$. The number of possible divisors of a given number $m$ is at most $\exp(O(\log m/\log \log m))$ \cite[Thm. $317$]{HaW}. Because the total number of acceptable partitions with at most $k$ parts is negligible ($x^{1 + o(1)}$), we assume that every partition has more than $k$ parts from this point.

Fix a tuple $(D_1, D_2, \ldots, D_k)$ of integers greater than $1$. We can bound the number of possible sums $a_1 + a_2 + \cdots + a_m$ with $m > k$ and $a_i/a_{i + 1} = D_i$ for all $i \leq k$. Note that
\begin{eqnarray*}
a_1 + a_2 + \cdots + a_m & \geq & a_1 + a_2 + \cdots + a_{k + 1} \\
& = & (D_1 D_2 \cdots D_k) a_{k + 1} + (D_2 D_3 \cdots D_k) a_{k + 1} + \cdots + D_k a_{k + 1} + a_{k + 1} \\
& = & (1 + D_k + D_{k - 1} D_k + \cdots + (D_1 D_2 \cdots D_k)) a_{k + 1}.
\end{eqnarray*}
If $a_1 + a_2 + \cdots + a_{k + 1} \leq x$, then $a_{k + 1} \leq x/(1 + D_k + D_{k - 1} D_k + \cdots + (D_1 D_2 \cdots D_k))$. Given $a_{k + 1}$ and $(D_1, \ldots, D_k)$, we can bound the number of possible values for the other $a_i$'s. If $i \leq k$, then $a_i$ is uniquely determined by $a_{k + 1}$ and the $D_j$'s. In addition, $a_{k + 1}, a_{k + 2}, \ldots, a_m$ uniquely determines a factorization of $a_{k + 1}$. So, there are $g(a_{k + 1})$ possible tuples given our constraints. Summing over all possible $a_{k + 1}$ gives us an upper bound of
\[G(x/(1 + D_k + D_{k - 1} D_k + \cdots + (D_1 D_2 \cdots D_k))).\]

We use a similar argument to obtain the lower bound. This time we observe that
\[a_1 + a_2 + \cdots + a_m \leq a_1 + a_2 + \cdots + a_{k + 1} + (a_{k + 1}/2) + (a_{k + 1}/4) + \cdots < a_1 + a_2 + \cdots + a_k + 2a_{k + 1}.\]
If $a_1 + a_2 + \cdots + a_k + 2a_{k + 1} \leq x$, then $a_1 + a_2 + \cdots + a_m \leq x$ as well. By an argument similar to the one for the lower bound, we obtain $G(x/(2 + D_k + D_{k - 1} D_k + \cdots + (D_1 D_2 \cdots D_k)))$ tuples. Setting $d_i = D_{k - i}$ and summing over all tuples $(d_1, d_2, \ldots, d_k)$ puts $B(x)$ between
\[\sum_{d_1, d_2, \ldots, d_k > 1} G\left(\frac{x}{2 + d_1 + d_1 d_2 + \cdots + (d_1 d_2 \cdots d_k)}\right)\]
and
\[\sum_{d_1, d_2, \ldots, d_k > 1} G\left(\frac{x}{1 + d_1 + d_1 d_2 + \cdots + (d_1 d_2 \cdots d_k)}\right).\]
To finish the proof, we simply use our asymptotic formula for $G(x)$ and factor out $x^\rho$.
\end{proof}

Unfortunately, I am unaware of any closed form for these sums. However, if $k = 1$, then we have the following simplifications:
\begin{eqnarray*}
\sum_{d > 1} \frac{1}{(1 + d)^\rho} & = & \sum_{d = 3}^\infty \frac{1}{d^\rho} = \zeta(\rho) - 1 - 2^{-\rho} = 1 - 2^{-\rho}, \\
\sum_{d > 1} \frac{1}{(2 + d)^\rho} & = & \sum_{d = 4}^\infty \frac{1}{d^\rho} = \zeta(\rho) - 1 - 2^{-\rho} - 3^{-\rho} = 1 - 2^{-\rho} - 3^{-\rho}.
\end{eqnarray*}
These sums imply that
\[-\frac{2(1 - 2^{-\rho} - 3^{-\rho})}{\rho \zeta'(\rho)} \leq c \leq -\frac{2(1 - 2^{-\rho})}{\rho \zeta'(\rho)}.\]

\section{Bounds on $a(n)$}

In this section, we write a few simple arguments bounding the maximal orders of $a(n)$ and $b(n)$. We also lay out a few problems for future research. From here on, we use the notation $f(x) \lesssim g(x)$ to mean $f(x) \leq (1 + o(1)) g(x)$ as $x \to \infty$.

Though we have an asymptotic formula for $A(x)$, we have few non-trivial results about the maximal order of $a(n)$. Clearly, $a(n) \in [A(n)/n, A(n)]$, which implies that
\[cx^{\rho - 1} \lesssim \max_{n \leq x} a(n) \lesssim cx^\rho,\]
\[(c/2)x^{\rho - 1} \lesssim \max_{n \leq x} b(n) \lesssim (c/2)x^\rho.\]

Estimates on the maximal order of $g(n)$ have a rich history. Hille \cite{Hi} initially showed that for any $\epsilon > 0$, there are infinitely many $n$ for which $g(n) > n^{\rho - \epsilon}$. This result has been improved numerous times \cite{Kal, E, Ev, KlLu}. Most recently, Del{\' e}glise, Hernane, and Nicolas \cite{DHN} showed that there exist positive constants $C_1$ and $C_2$ such that
\[x^\rho \exp\left(-C_1 \frac{(\log x)^{1/\rho}}{\log \log x}\right) < \max_{n \leq x} g(n) < x^\rho \exp\left(-C_2 \frac{(\log x)^{1/\rho}}{\log \log x}\right)\]
for all sufficiently large $x$. They also conjecture that there exists a positive constant $C$ such that
\[\max_{n \leq x} g(n) = x^\rho \exp\left(-(C + o(1)) \frac{(\log x)^{1/\rho}}{\log \log x}\right).\]

Chor, Lemke, and Mador \cite{ChLM} found an explicit bound for $g(n)$. A few years later, Coppersmith and Lowenstein \cite{CoLo} found an elementary proof of this bound, which we rewrite below.

\begin{thm} For all $n$, we have $g(n) \leq n^\rho$.
\end{thm}

\begin{proof} We proceed by induction. We already know that $g(1) = 1^\rho = 1$. Suppose $n > 1$. Every factorization of $n$ has the form $a_1 \cdot a_2 \cdots a_k$, where $a_1$ is a divisor of $n$ greater than $1$. The number of tuples $(a_2, \ldots, a_k)$ is equal to $g(n/a_1)$. Therefore,
\[g(n) = \sum_{\substack{d | n \\ d > 1}} g\left(\frac{n}{d}\right) \leq \sum_{\substack{d | n \\ d > 1}} \left(\frac{n}{d}\right)^\rho = n^\rho \sum_{\substack{d | n \\ d > 1}} \frac{1}{d^\rho} < n^\rho \sum_{d = 2}^\infty \frac{1}{d^\rho} = n^\rho (\zeta(\rho) - 1) = n^\rho. \qedhere\]
\end{proof}

Maximal orders for \emph{unordered} factorizations have a rich history as well. Oppenheim \cite{O} found a slightly erroneous asymptotic formula for the maximal order, which Canfield, Erd{\H o}s, and Pomerance \cite{CEP} later corrected and refined. In addition, Mattics and Dodd \cite{MatDo} later showed that the number of unordered factorizations of $n$ is less than $n/\log n$ for all $n \neq 1, 144$, resolving a conjecture of Hughes and Shallit \cite{HuS}.

The function $b(n)$ has a similar recurrence relation to $g(n)$. The base case is $b(1) = 1$. For larger $n$, we have the following. Let $n = a_1 + a_2 + \cdots + a_k$. If we factor out $a_k$ and subtract $1$, we obtain $(n/a_k) - 1 = (a_1/a_k) + (a_2/a_k) + \cdots + (a_{k - 1}/a_k)$. For a given $a_k$, the number of such sums is $b((n/a_k) - 1)$. Setting $d = a_k$ and summing over all possible $d$ gives us
\[b(n) = \sum_{\substack{d | n \\ d > 1}} b\left(\frac{n}{d} - 1\right).\]
Applying Coppersmith and Lowenstein's proof gives us $b(n) \leq n^\rho$ and $a(n) \leq 2n^\rho$. Interestingly, $a(n)$ satisfies the recurrence
\[a(n) = \sum_{d | n} a\left(\frac{n}{d} - 1\right)\]
with $a(0) = 1$.

As for minimal orders, we simply observe that $b(p) = 1$ for all primes $p$ simply because the only possible sum is $p$ itself. Unfortunately, this tells us nothing about the minimal order of $a(n)$. Given that $b(n)$ depends on the divisors of $n - 1$ and not the divisors of $n$, one should expect that it has significantly less variation that $g(n)$. In light of this fact, we propose the following.

\begin{conj} As $n \to \infty$, we have $b(n) \asymp n^{\rho - 1}$ for composite $n$ and $a(n) \asymp n^{\rho - 1}$ for all $n$.
\end{conj}

\end{document}